\documentclass{amsart}
\usepackage{graphicx}

\usepackage{amsthm}

\usepackage{verbatim}
\usepackage[curve,matrix,arrow]{xy}
\usepackage{amsmath,amsfonts}
\usepackage{epsf}

\usepackage{amssymb, graphics}

\newtheorem{thm}{Theorem}[section]

\newtheorem{lem}[thm]{Lemma}

\theoremstyle{definition}

\theoremstyle{remark}

\numberwithin{equation}{section}

\def\RR{{\mathbb R}}
\def\CC{{\mathbb C}}

\newcommand{\thmref}[1]{Theorem~\ref{#1}}

\newcommand{\lemref}[1]{Lemma~\ref{#1}}
\newcommand{\eqnref}[1]{Equation~(\ref{#1})}

\newcommand{\secref}[1]{\S\ref{#1}}

\begin{document}

\Large\title[The determinant of toeplitz matrices]{A fast elementary algorithm for computing the determinant of toeplitz matrices}

\author{Zubeyir Cinkir}
\address{Zubeyir Cinkir\\
Department of Mathematics\\
Zirve University\\
27260, Gaziantep, TURKEY\\}
\email{zubeyir.cinkir@zirve.edu.tr}


\begin{abstract}
In recent years, a number of fast algorithms for computing the determinant of a Toeplitz matrix were developed.
The fastest algorithm we know so far is of order $k^2\log{n}+k^3$, where $n$ is the number of rows of the Toeplitz matrix and $k$ is the bandwidth size.
This is possible because such a determinant can be expressed as the determinant of certain parts of $n$-th power of a related $k \times k$ companion matrix. In this paper, we give a new elementary proof of this fact, and provide various examples. We give symbolic formulas for the determinants of Toeplitz matrices in terms of the eigenvalues of the corresponding companion matrices when $k$ is small.
\end{abstract}

\maketitle

\section{Introduction} \label{sec intr}
\vskip .1 in


In this paper, we consider an $n \times n$ Toeplitz band matrix $T_{n}$ with $r$ and $s$ superdiagonals as shown below:
\begin{eqnarray*}\label{eqn mat Tn}
T_{n}=
\left(
\begin{array}{ccccccccc}
 a_0     & a_1    & \cdots & a_s &     &        & \mathbf{0} \\
 a_{s+1} & a_0    & \cdots &     & a_s &        &           \\
 \vdots  & \ddots & \ddots &     &     & \ddots &           \\
a_{s+r}  &        &        &     &     &        & a_s       \\
         & a_{s+r}&        &     &     &        & \vdots     \\
         &        & \ddots &     &     &        &   a_1  \\
\mathbf{0}&        &        & a_{s+r}& \cdots &  a_{s+1}  & a_0 \\
\end{array}
\right)_{n \times n},
\end{eqnarray*}
where $r+s=k$, $a_s \neq 0$, $a_{s+r} \neq 0$, $a_i \in K$ for $i=0, \dots, k$ and $K$ is a field. Without loss of generality we assume that $s \leq k$, as the determinant remains the same under taking the transpose, and our main concern is the determinant of $T_{n}$.

Finding fast algorithms to compute $\det(T_{n})$ is of interest for various applications. A number of fast algorithms computing $\det(T_{n})$ are developed recently (for tri-diagonal and pentadiagonal cases, see \cite{ZC}, \cite{KM}, \cite{M}, \cite{LHL}, \cite{S} ). Whenever $r=s=2$, the author \cite{ZC} gave an elementary algorithm computing $\det(T_{n})$ in $82 \sqrt{n}+O(\log{n})$ operations. In this paper, we both improved and generalized the algorithm given in \cite{ZC}. Namely, the algorithm we give here works for any $r \geq 1$ and $s \geq 1$, i.e., for any $k \geq 2$. Moreover, it takes
$O(\frac{3}{2}k^2 \log_2 \frac{n}{k}+s^3)$ operations to compute $\det(T_{n})$. The key part in this improvement is that the computation of $\det(T_{n})$ can be related to the powers of the following $k \times k$ companion matrix $C$ associated to $T_{n}$:
\begin{eqnarray}\label{eqn mat C}
C=\left(
\begin{array}{ccccc}
 \frac{-a_{s-1}}{a_s} & & &  \\
 \vdots & & & \\
 \frac{-a_1}{a_s} &  & & & \\
 \frac{-a_0}{a_s} & & I_{(k-1) \times (k-1)} &  \\
 \frac{-a_{s+1}}{a_s} &  &  & \\
 \vdots & & & \\
 \frac{-a_{s+r}}{a_s} & 0  & \cdots & 0
\end{array}
\right)_{k \times k},
\end{eqnarray}
where $I_{(k-1) \times (k-1)}$ is the identity matrix of size $(k-1) \times (k-1)$.
The characteristic polynomial $ch_C(x)$ of $C$ is given by
\begin{equation*}\label{eqn characteristic pol}
    \det(x I-C)=x^k+\frac{a_{s-1}}{a_s} x^{k-1}+\cdots +\frac{a_0}{a_s} x^{k-s}+\frac{a_{s+1}}{a_s} x^{r-1}+\cdots + \frac{a_{s+r-1}}{a_s}x+\frac{a_{s+r}}{a_s}.
\end{equation*}
Note that the algorithm we give in this paper is the fastest known algorithm to compute $\det(T_{n})$, and was first given in \cite{BP}.
Our method is elementary and more clear. More precisely, we give a new proof of the following theorem:
\begin{thm}[\thmref{thm main3} in \secref{sec proof}]\label{thm main3n}
Let $M$ be the upper left $s \times s$ submatrix of $C^{n}$.
For every integer $n \geq k$,
$\det(T_{n}) = (-1)^{ns} a_s^{n} \cdot \det(M).$
\end{thm}
After this theorem, our attention focuses on the computations of powers of $C$. Thus, we give brief discussion of various methods of computing $C^{n}$ in \secref{sec comp Cn}.

Closed form formulas for $\det(T_{n})$ can be given in terms of the roots of $ch_C(x)$ as previously described in \cite{T}. In \secref{sec pentadiagonals}, we illustrate how this is possible for pentadiagonal Toplitz matrices, and give explicit formulas.

The method given here is effective as long as $k$, the number of nonzero diagonals of $T$, is not close to $n$ which is the number of rows of  $T$.

\section{A fast algorithm for computing $\det(T_n)$}\label{sec proof}

In this section, we give an elementary algorithm for computing $\det(T_{n})$,
which is the fastest known algorithm so far.

First, we describe the outline of the algorithm as follows.
%
We move the first $s$ column vectors of $T_n$ and make them the last vectors, successively. This gives a matrix $P$. If
the column vectors of $T_n$ are $\{ C_1, C_2, \ldots, C_n \}$, then the column vectors of $P$ are
$\{ C_{s+1}, C_{s+2}, \ldots, C_n, C_{1}, \ldots, C_{s} \}$. Note that $\det(T_{n})=(-1)^{(n-1)s}\det(P)$.
Then we multiply the first column by suitable terms and add to the last $s$ columns of $P$ so that the only nonzero entry in the first row will be $a_s$. If the resulting matrix is $P'$, $\det(P)$ is nothing but $a_s$ times the determinant of the cofactor $P_{1,1}'$ of $P'$. We note that  $P_{1,1}'$ is a matrix of size $(n-1)\times (n-1)$, and it is of similar form as $P$. Following the same procedure applied to $P$, we relate $\det(P_{1,1}')$ to the determinant of a matrix of size $(n-2)\times (n-2)$. Continuing in this way, the problem of computing $\det(T_{n})$ can be reduced to the computation of the determinant of a $k \times k$ matrix, and this matrix can be computed easily as it is $n$-th power of a $k \times k$ companion matrix. Next, we describe this algorithm in detail. To be precise, we introduce some notation and deduce some results.

Let $f:\RR^k \longrightarrow \RR^k$ be a linear transformation given by $f([x_1, \, x_2, \, \ldots, \, x_k]^t)
=[x_2-x_1 \frac{a_{s-1}}{a_s}, \, x_3-x_1 \frac{a_{s-2}}{a_s}, \, \ldots, \, x_{s+1}-x_1 \frac{a_{0}}{a_s}, \, x_{s+2}-x_1 \frac{a_{s+1}}{a_s}, \, \ldots, \, x_{k}-x_1 \frac{a_{k-1}}{a_s}, \, -x_1 \frac{a_{k}}{a_s}]^t$, where $v^t$ is the transpose of a vector $v$. Let $F$ be a map sending a $k \times s$ matrix $A$ to another $k \times s$ matrix $F(A)$ by applying $f$ to every column of $A$. That is, if the columns of $A$ are $\{ C_1, C_2, \ldots, C_s \}$, then the columns of $F(A)$ are $\{ f(C_1), f(C_2), \ldots, f(C_s) \}$.

We define a sequence of $k \times s$ matrices $(A_{i})_{i\geq 0}$ recursively by setting
$A_{i+1}=F(A_i)$ for every integer $i \geq 0$, and by taking the following initial value:
\begin{eqnarray*}\label{ eqn mat A0}
A_{0}=\left(
\begin{array}{ccccc}
a_0 & a_1 & \cdots & a_{s-2}& a_{s-1} \\
a_{s+1} & a_0 & a_1 & \cdots & a_{s-2} \\
\vdots & \ddots& \ddots & \ddots & \vdots \\
a_{s+r} & & \ddots &  \ddots & \vdots \\
 & \ddots & & \ddots & a_0\\
& & \ddots &  & a_{s+1}\\
& & & \ddots & \vdots\\
\mathbf{0} & & &  &a_{s+r}
\end{array}
\right)_{k \times s}.
\end{eqnarray*}

Let $\{v_1, \, v_2, \, \ldots, \, v_k \}$ be the standard basis, consisting of column vectors, of $\RR^k$. Note that
$f(v_1)=[\frac{-a_{s-1}}{a_s}, \, \ldots, \,  \frac{-a_1}{a_s}, \, \frac{-a_0}{a_s}, \, \frac{-a_{s+1}}{a_s}, \, \ldots, \,  \frac{-a_{s+r}}{a_s}]^t$
and $f(v_i)=v_{i-1}$ for each $i=2,\, 3, \, \ldots, \, k$. Thus, the matrix $C$ given in \eqnref{eqn mat C} is nothing but the matrix representation of the linear transformation $f : \RR^k \longrightarrow \RR^k$.

\begin{lem}\label{lem product mat}
Let the transformation $F$ and the matrices $A_{i}$ and $C$ be as defined before.
We can express $F$ in terms of $C$ as follows:
$\displaystyle{F(A_{i})=C \cdot A_{i}.}$
Therefore, for any integer $i \geq 1$ we have
$A_{i}=C^i \cdot A_{0}$.
\end{lem}
\begin{proof}
Note that $C \cdot A_{i}$ is nothing but $A_{i+1}$.
Thus, the proof of the first part follows from the definition of $F$.
Applying the first part successively to $A_{0}$ gives
$F^i(A_{0})=C^i \cdot A_{0}$. On the other hand, $F^i(A_{0})=A_{i}$. This completes the proof.
\end{proof}

Let $O_{(n-k) \times s}$ be the zero matrix of size $(n-k) \times s$. We set
\begin{eqnarray*}\label{eqn mat Bn}
\qquad \, B_{n}=\left(
\begin{array}{ccc}
a_s & & \mathbf{0} \\
a_{s-1} & a_s & \\
\vdots & \ddots& \ddots \\
a_{0} & & \ddots \\
a_{s+1} & a_0 & \\
\vdots & \ddots & \ddots \\
a_{s+r-1} & & \ddots\\
a_{s+r} & a_{s+r-1} &\\
& \ddots & \ddots \\
\mathbf{0} & &\ddots
\end{array}
\right)_{n \times (n-s)}, \quad
P_{n,i}=\left(
\begin{array}{c|c}
& A_i \\ \cline{2-2}
B_n & \\
 & O_{(n-k) \times s}
\end{array}
\right)_{n \times n}.
\end{eqnarray*}
Next, we relate the determinants of $T_{n}$ and $P_{n,0}$.
\begin{lem}\label{lem Tn and Pn0}
For any $n \geq k$ and $s \geq 0$, $\det(T_{n})=(-1)^{(n-1)s}\det(P_{n,0})$.
\end{lem}
\begin{proof}
We move the first $s$ column vectors of $T_n$ and make them the last vectors, successively. This gives the matrix $P$. More precisely, if the column vectors of $T_n$ are $\{C_1, C_2, \ldots, C_n \}$, then the column vectors of $P$ are
$\{ C_{s+1}, C_{s+2}, \ldots, C_n, C_{1}, \ldots, C_{s} \}$.
Note that the matrix $P$ is nothing but $P_{n,0}$. Therefore, $\det(T_{n})=(-1)^{(n-1)s}\det(P_{n,0})$.
\end{proof}

The following theorem is a generalization of \cite[Theorem 2.3]{ZC}:
\begin{thm}\label{thm old main1}
Let $s$ and $k$ be as before.
For every integers $n \geq k$ and $i$ such that $0 \leq i \leq n-k$,
we have
$$\det(T_{n}) = (-1)^{(n-1)s} a_{s}^{i} \cdot \det(P_{n-i,i}).$$
In particular,
$$\det(T_{n}) = (-1)^{(n-1)s} a_{s}^{n-k} \cdot \det(P_{k,n-k}).$$
\end{thm}
\begin{proof}
By \lemref{lem Tn and Pn0}, $\det(T_n)=(-1)^{(n-1)s} \det(P_{n,0})$.
Therefore, we are done if $n=k$ or $i=0$. For the rest of the proof, we assume $n>k$ and $i>0$.

%
For any integer $i$ with $1 \leq i \leq n-k$,
we multiply the first column of $P_{n-i+1,i-1}$ by $\frac{-b_j}{a_s}$ and add it to its $(n-i-s+j+2)$-th column for each $j=0, \, 1, \, \ldots, \, s-1$, where $b_j$ is the $(1,j+1)$-th entry of $A_{i-1}$. Let $R_i$ be the resulting matrix.
Clearly, $\det(P_{n-i+1,i-1})=\det(R_i)$.
The only nonzero entry in the first row of $R_i$ is the $(1,1)$-th entry with value $a_s$.
Note that $(1,1)$-th minor of $R_i$ is nothing but $P_{n-i,i}$. Expanding the determinant of
$R_i$ using the first row gives
$\det(R_i)=a_s \cdot \det(P_{n-i,i})$. Thus, $\det(P_{n-i+1,i-1})=a_s \cdot \det(P_{n-i,i})$.

This gives that $\det(P_{n,0}) = a_{s}^{i} \cdot \det(P_{n-i,i})$ for each integer $i$ with $0 \leq i \leq n-k$.
Then the result follows.

%
%
\end{proof}

\begin{thm}\label{thm main1}
For every integer $n \geq k$,
we have
$$\det(T_{n}) = (-1)^{(n-1)s} a_s^{n-k} \cdot \det([B_k,C^{n-k} \cdot A_0]).$$
\end{thm}
\begin{proof}
By the definition $P_{k,n-k}=[B_k \, | \, A_{n-k}]$. On the other hand, $A_{n-k}=C^{n-k} \cdot A_0$ by \lemref{lem product mat} with $i=n-k$. Then the result follows from the second part of \thmref{thm old main1}.
\end{proof}
We can improve \thmref{thm main1} as follows:
\begin{thm}\label{thm main2}
Let $M$ be the upper $s \times s$ submatrix of $C^{n-s} \cdot A_0$.
For every integer $n \geq k$,
we have
$$\det(T_{n}) = (-1)^{(n-1)s} a_s^{n-s} \cdot \det(M).$$
\end{thm}
\begin{proof}
The proof follows by similar arguments given in the proof of \thmref{thm old main1}, but we need to introduce new notation for precise and shorter description.

We have
\begin{equation}\label{eqn Tn and Bk}
\begin{split}
\det(T_{n}) = (-1)^{(n-1)s} a_s^{n-k} \cdot \det([B_k \, | \, A_{n-k}])
\end{split}
\end{equation}
by the second part of \thmref{thm old main1} and the proof of \thmref{thm main1}.

Let $U_i$ denote the lower right $(k-i \times r-i)$ submatrix of $B_k$, and let $V_i$ denote the upper $k-i \times s$ submatrix of
$A_{n-k+i}$ for $i=0, \, 2, \, \ldots, \, r$. Note that  $[U_r \, | \, V_{r}]=V_{r}$, the upper $s \times s$ submatrix of $A_{n-s}$, and $[U_0 \, | \, V_{0}]=[B_k \, | \, A_{n-k}]$.

For any integer $i$ with $0 \leq i < r$,
we multiply the first column of $[U_i \, | \, V_{i}]$ by $\frac{-b_j}{a_s}$ and add it to its $(r-i+j+1)$-th column for each $j=0, \, 1, \, \ldots, \, s-1$, where $b_j$ is the $(1,j+1)$-th entry of $V_{i}$. Let $S_i$ be the resulting matrix.
Clearly, $\det([U_i \, | \, V_{i}])=\det(S_i)$.
The only nonzero entry in the first row of $S_i$ is the $(1,1)$-th entry with value $a_s$.
Note that $(1,1)$-th minor of $S_i$ is nothing but $[U_{i+1} \, | \, V_{i+1}]$. Expanding the determinant of
$S_i$ using the first row gives
$\det(S_i)=a_s \cdot \det([U_{i+1} \, | \, V_{i+1}])$. Thus, $\det([U_{i} \, | \, V_{i}])=a_s \cdot \det([U_{i+1} \, | \, V_{i+1}])$.

This gives that $\det([U_{0} \, | \, V_{0}])=a_s^r \cdot \det([U_{r} \, | \, V_{r}])$. That is,
$\det([B_k \, | \, A_{n-k}])=a_s^r \cdot \det(V_r)$. Combining this with \eqnref{eqn Tn and Bk}, we see that
$\det(T_{n}) = (-1)^{(n-1)s} a_s^{n-s} \cdot \det(M)$, where $M$ is the upper $s \times s$ submatrix of $A_{n-s}$.
On the other hand,  $A_{n-s}=C^{n-s} \cdot A_0$ by \lemref{lem product mat}. Then the result follows.
\end{proof}

Note that \thmref{thm main2} outlines a fast algorithm to compute $\det(T_{n})$. However, we look for even simpler formula, which is possible if we find an answer to the following question:

Can we get a simpler expression for the determinant of the upper $s \times s$ submatrix of $C^{n-s} \cdot A_0$?

\lemref{lem formula for CA0} below provides a fairly simple expression for this determinant.
First, we give some well-known facts that we will use.

For each integer $i=1, \, 2, \, \dots \, k$, suppose $u_{i,n}$ satisfies the recurrence relation
$$x_{n+k}= -\frac{a_{s-1}}{a_s} x_{n+k-1}-\cdots -\frac{a_0}{a_s} x_{n+k-s}-\frac{a_{s+1}}{a_s} x_{n+r-1}-\cdots - \frac{a_{s+r-1}}{a_s}x_{n+1}-\frac{a_{s+r}}{a_s}x_{n},$$ and that their initial values are given by the first equality below, where $I_{k \times k}$ is the identity matrix of size $k \times k$. Then the powers of $C$ are given by the second equality below:
\begin{eqnarray*}
\left(
\begin{array}{cccc}
u_{1,k-1} & u_{1,k-2} & \cdots & u_{1,0} \\
u_{2,k-1} & u_{2,k-2} & \cdots & u_{2,0} \\
 \vdots & \vdots & \cdots & \vdots \\
u_{k,k-1} & u_{k,k-2} & \cdots & u_{k,0}
\end{array}
\right)
= I_{k \times k}, \quad
C^n=\left(
\begin{array}{cccc}
 u_{1,n+k-1} & u_{1,n+k-2} & \cdots & u_{1,n} \\
u_{2,n+k-1} & u_{2,n+k-2} & \cdots & u_{2,n} \\
\vdots & \vdots & \cdots & \vdots \\
u_{k,n+k-1} & u_{k,n+k-2} & \cdots & u_{k,n}
\end{array}
\right).
\end{eqnarray*}

\begin{lem}\label{lem formula for CA0}
For every positive integers $n$ and $s$ with $n \geq s+1$, the determinant of the upper $s \times s$ submatrix of  $C^{n-s} \cdot A_0$ is $(-a_s)^s$ times the determinant of the upper left $s \times s$ submatrix of $C^{n}$.
\end{lem}
\begin{proof}
Let $c_{l,m}$ be the $(l,m)$-th entry of the matrix $C^{n-s} \cdot A_0$ of size $k \times s$. Since $c_{l,m}$ is $l$-th row of $C^{n-s}$ times $m$-th column of $A_0$, we have
$$c_{l,m}=\sum_{j=1}^{m}a_{m-j} u_{l,n+r-j}+\sum_{i=1}^{r}a_{s+i}u_{l,n+r-m-i}.$$
Using the recursive formula of $u_{l,n+k-m}$, we can express the above equality as follows:
$$c_{l,m}=-\sum_{i=0}^{s-m}a_{s-i} u_{l,n+k-m-i}.$$
Now, we note that $m$-th column of $C^{n-s} \cdot A_0$ contains linear combinations of its last $s-m$ columns, and that the determinant does not change by adding a multiple of a column to some other column.

Let $M_0$ be the upper $s \times s$ submatrix of  $C^{n-s} \cdot A_0$. For every $t=0, \, \dots s-2$, suppose $M_{t+1}$ is the matrix obtained from $M_t$ by multiplying its $(s-t)$-th column by
$-\frac{a_{s-j}}{a_s}$ and then adding to its $(s-t-j)$-th column for each $j=1, \, 2, \dots, \, s-t-1$. Note that $\det(M_{0})=\det(M_{s-1})$ and that $(l,m)$-th entry of $M_{s-1}$ is $-a_{s} u_{l,n+k-m}$. That is, $\det(M_{s-1})$ is $(-a_s)^s$ times the determinant of a matrix with $(l,m)$-th entry $u_{l,n+k-m}$, which is nothing but the upper left $s \times s$ submatrix of $C^{n}$. This completes the proof.
\end{proof}

Next, we give the main result of this paper. Namely, the computation of $\det(T_{n})$ is basically reduced to the computation of $n$-th power of the $k \times k$ matrix $C$, and the computation of the determinant of an $s \times s$ matrix:
\begin{thm}\label{thm main3}
Let $M$ be the upper left $s \times s$ submatrix of $C^{n}$.
For every integer $n \geq k$,
$$\det(T_{n}) = (-1)^{ns} a_s^{n} \cdot \det(M).$$
\end{thm}
\begin{proof}
The result follows from \thmref{thm main2} and \lemref{lem formula for CA0}.
\end{proof}
Since $M$ is $s \times s$ matrix, its determinant can be computed in $O(s^3)$ operations. In fact, the power $3$ here can be taken less. More costly part is the computation of $C^n$, which can be done in $O(\frac{3}{2}k^2 \log_2 \frac{n}{k})$ operations as explained at the end of \secref{sec comp Cn}. Therefore, $\det(T_{n})$ can be computed in $O(\frac{3}{2}k^2 \log_2 \frac{n}{k}+s^3)$ operations.

Note that \thmref{thm main3} is very similar to \cite[Prop 2.1]{BP}. However, the proof we gave here is more elementary and clear.

We can use the algorithm described in this section to compute the characteristic polynomial of $T_n$ if we start with $T_{n}-\lambda I$ rather than $T_n$. Therefore,
\thmref{thm main3} can be extended as follows:
\begin{thm}\label{thm character Tn}
Let $C_{\lambda}$
be the matrix obtained from $C$
by replacing $a_0$ by $a_0-\lambda$.
For every integer $n \geq k$,
we have
$$\det(T_{n}-\lambda I) = (-1)^{ns} a_s^{n} \cdot \det(M_{\lambda}),$$
where $I$ is the $n \times n$ identity matrix and $M_{\lambda}$ is the upper left $s \times s$ submatrix of $C_{\lambda}^{n}$.
\end{thm}

\section{Computing $C^n$}\label{sec comp Cn}


In this section, we briefly describe the several ways to compute $C^{n}$ from $C$.

We set $C^{0}=I$, where $I$ is the identity matrix of size $k \times k$. Then we use $C^{n}= C^{n/2} C^{n/2}$ if $n$ is even, and we use $C^{n}= C C^{n-1}$ if $n$ is odd. In this way, $C^{n}$ can be computed in $O(\log n)$ times the number of operations to compute the product of two $k \times k$ matrices. If $k$ is large, one should consider other methods of computing $C^{n}$.

As shown in the previous section, $C^n$ can be expressed in terms of $k$ homogeneous linear recurrences $u_{i,j}$, where $i \in \{1, \, 2, \, \cdots, \, k \}$. If we compute the roots of the characteristic polynomial of $C$, $ch_C(x)$, we can express each of these recurrences in terms of these roots since they satisfy the recurrence relation given by $ch_C(x)$ and their initial values are known (see \cite[section 7.2.9]{SB} for solutions of such recurrences). We should note that the computation of roots of a polynomial of degree higher than five can be a difficult task.

Another method is to use Jordan canonical form of $C$. We explain the advantageous steps in this approach.
First, we note that the characteristic polynomial of $C$, $ch_C(x)$, is the minimal polynomial of $C$ (see \cite[page 325]{SB}). Over $\CC$, $ch_C(x)$ factors into linear terms. Then $C$ have a Jordan canonical form $J$ such that
$C=V^{-1} J V$,
for the Vandermonde matrix $V$ as shown in \eqnref{eqn Jordan} (if any of the eigenvalues has multiplicity more than one, then $V$ will be the corresponding confluent Vandermonde matrix). Then $C^{n}=V^{-1} J^{n} V$.
%
%
This along with \thmref{thm main3} implies that $\det(T_{n}) = (-1)^{ns} a_{s}^{n} \cdot \det(M)$, where $M$ is the upper left $s \times s$ submatrix of $V^{-1} J^{n} V$.
Moreover, we have the following useful information about the part $J^{n}$:

If $J=\text{diag}(J_{1}, J_2, \dots, J_s )$ with $s \leq k$ and $J_i$ is Jordan block matrix for $1 \leq i \leq k$, then
$J^{n}=\text{diag}(J_{1}^n, J_2^n, \dots, J_s^n )$ for every integer $n \geq 1$.

Let $J_{(i)}(\lambda)$ be a Jordan block matrix, part of a Jordan matrix $J$, corresponding to an eigenvalue $\lambda$. As shown in \eqnref{eqn Jordan}, $J_{(i)}(\lambda)$ is an upper triangular square matrix of size $i \times i$ and having $\lambda$'s on its diagonal, $1$'s on the next diagonal and $0$'s elsewhere.
\begin{eqnarray}\label{eqn Jordan}
\qquad J_{(i)}(\lambda)=
\left(
\begin{array}{ccccc}
 \lambda & 1 & &  &  \\
  & \lambda & 1 & &  \\
  &  & \lambda &  \ddots & \\
  &  & & \ddots & 1 \\
  & &  &  & \lambda
\end{array}
\right)_{i \times i}, \, \, \,
V=
\left(
\begin{array}{ccccc}
 \lambda_1^{k-1} & \lambda_1^{k-2} & \cdots & \lambda_1 & 1 \\
 \lambda_2^{k-1} & \lambda_2^{k-2} & \cdots & \lambda_2 & 1 \\
  \vdots & \vdots  & \cdots &  \vdots & \vdots \\
 \lambda_k^{k-1} & \lambda_k^{k-2} & \cdots & \lambda_k & 1
\end{array}
\right)_{k \times k}.
\end{eqnarray}
Then one can show that \cite[Exercise 7.3.27]{W} its $n$-th power is as follows:
\begin{eqnarray*}\label{eqn Jordan power}
\quad J_{(i)}^n(\lambda)=
\left(
\begin{array}{ccccc}
 \lambda^n & n \lambda^{n-1}  & \frac{n(n-1)}{2} \lambda^{n-2} & \cdots &  \\
  & \lambda^n & n \lambda^{n-1} & \ddots &  \\
  &  & \lambda^n &  \ddots & \frac{n(n-1)}{2} \lambda^{n-2} \\
  &  & & \ddots & n \lambda^{n-1} \\
  & &  &  & \lambda^n
\end{array}
\right),
\end{eqnarray*}
where $J_{(i)}^n(\lambda)$ has $n+1$ nonzero diagonals (or less if the dimension of $J_{(i)}(\lambda)$ is less),
and the $j$-th diagonal from the main diagonal has entries $\binom{n}{j}\lambda^{n-j}$.

Note that various methods of matrix exponentiation are explained in the article  \cite{ML}. One can use the modified versions of those methods to compute $C^n$. The reader can find more information on the critique of those methods in the article \cite{ML} and the references therein.

One can use combinatorial arguments to derive combinatorial formulas for the entries of any power of $C$ (see \cite{CL} and \cite{LD}). Although such combinatorial formulas are very interesting, they don't lead to fast algorithms to compute powers of $C$.

Assuming that the roots of $ch_C(x)$ are known, one can compute $C^n$ in $O(k \log_2 n)$ operations by using the method proposed by W. Trench \cite{T}. D. Bini and V. Pan \cite[Page 436]{BP} improved Trench's result. Namely, without assuming that the roots of $ch_C(x)$ are known, one can compute $C^n$ in $O(\frac{3}{2}k^2 \log_2 \frac{n}{k})$ operations. If the field that $C$ defined over supports Fast Fourier Transform, Bini and Pan showed that $C^n$ can be computed in $O(k \log k \log_2 \frac{n}{k})$ operations (see \cite[Page 436]{BP} and the references therein).

\section{Determinants of tridiagonal Toeplitz matrices}

In this section, we consider the tridiagonal matrices, i.e., we have $s=1$, $r=1$, and so $k=2$,
\begin{eqnarray*}\label{eqn mat tridiag Tnk2}
T_{n}=
\left(
\begin{array}{cccc}
 a &        b        &        &\mathbf{0} \\
 c &        a        &  \ddots    &   \\
   & \ddots & \ddots & b \\
\mathbf{0}& &    c   &   a
\end{array}
\right)_{n \times n}
\text{  and  } \, \, \,
C=\left(
\begin{array}{cc}
 -\frac{a}{b} & 1  \\
 -\frac{c}{b} & 0
\end{array}
\right),
\end{eqnarray*}
where $b \neq 0, \, c \neq 0$. In this case, the characteristic polynomial of $C$ is
$\det(x I-C)= x^2+\frac{a}{b}x+\frac{c}{b}$.

We note that
$
C^n=\left(
\begin{array}{cc}
u_{n+1} & u_{n}  \\
w_{n+1} & w_{n}
\end{array}
\right),
$
where both $u_n$ and $w_n$ satisfy the recursive relation $x_{n+2}=\frac{-a}{b} x_{n+1}- \frac{c}{b}x_n$ for each $n \geq 0$, and they have the initial values
$u_0=0$, $u_1=1$, $w_0=1$, $w_1=0$. We recognize that $u_n$ is nothing but the Lucas sequence of the first type. That is, $u_n=U_{n}(P,Q)$ for each integer $n \geq 0$, where $P=\frac{-a}{b}$ and $Q=\frac{c}{b}$. Because $U_{n+2}(P,Q)=P U_{n+1}- Q U_n$ for each integer $n \geq 0$, $U_0(P,Q)=0$ and $U_1(P,Q)=1$.

Applying \thmref{thm main3} gives
$$\det(T_{n})=  (-1)^{n} b^{n} U_{n+1} (\frac{-a}{b},\frac{c}{b}).$$
Hence, existing formulas for $U_n$ applies to  $\det(T_{n})$. For example, using the binet formula for $U_{n+1}$ we obtain
%
%
%
%

$$\det(T_{n})= \frac{1}{\sqrt{a^2-4bc}} \Big[ \big( \frac{a+\sqrt{a^2-4bc}}{2} \big)^{n+1}-\big( \frac{a-\sqrt{a^2-4bc}}{2} \big)^{n+1} \Big],$$
and using the recursive formula for $U_{n+1}$ we obtain
$$\det(T_{n+2})=a \cdot \det(T_{n+1})-bc \cdot \det(T_{n}), \text{   for $n \geq 2$},$$
where $\det(T_{1})= a$ and $\det(T_{2})=a^2 - b c$.

\section{Determinants of pentadiagonal Toeplitz matrices}\label{sec pentadiagonals}

In this section, we give explicit formulas for determinants of pentadiagonal matrices.

In this case, we have $s=2$, $r=2$, and so $k=4$,
\begin{eqnarray*}\label{eqn mat tridiag Tn4}
T_{n}=
\left(
\begin{array}{ccccc}
 a &        b        &   c     & &\mathbf{0} \\
 d &        a        &  \ddots    & \ddots    \\
 e & \ddots & \ddots &  \ddots & c \\
   & \ddots & \ddots &  a & b \\
\mathbf{0}& &    e   &  d & a
\end{array}
\right)_{n \times n}
\text{  and  } \, \, \,
C=\left(
\begin{array}{cccc}
 -\frac{b}{c} & 1 & 0 & 0 \\
 -\frac{a}{c} & 0 & 1 & 0 \\
 -\frac{d}{c} & 0 & 0 & 1 \\
 -\frac{e}{c} & 0 & 0 & 0
\end{array}
\right),
\end{eqnarray*}
where $ c \neq 0, \, e \neq 0$. The characteristic polynomial of $C$ is
$ch_C(x)=\det(x I-C)= x^4+\frac{b}{c}x^3+\frac{a}{c}x^2+\frac{d}{c}x+\frac{e}{c}$.

For each integer $n \geq 1$, $C^n$ is expressed in terms of
$u_n$, $v_n$, $t_n$ and $w_n$, each of which satisfies the recursive relation
$x_{n+4}=-\frac{b}{c}x_{n+3}-\frac{a}{c}x_{n+2}-\frac{d}{c}x_{n+1}-\frac{e}{c}x_{n}$, and that their initial values are given by the second equality below:
\begin{eqnarray*}\label{eqn mat tridiag Cn}
C^n=\left(
\begin{array}{cccc}
 u_{n+3} & u_{n+2} & u_{n+1} & u_{n} \\
 v_{n+3} & v_{n+2} & v_{n+1} & v_{n} \\
 t_{n+3} & t_{n+2} & t_{n+1} & t_{n} \\
 w_{n+3} & w_{n+2} & w_{n+1} & w_{n}
\end{array}
\right),
\text{and} \, \, \,
\left(
\begin{array}{cccc}
 u_{3} & u_{2} & u_{1} & u_{0} \\
 v_{3} & v_{2} & v_{1} & v_{0} \\
 t_{3} & t_{2} & t_{1} & t_{0} \\
 w_{3} & w_{2} & w_{1} & w_{0}
\end{array}
\right)
=
\left(
\begin{array}{cccc}
1 & 0 & 0 & 0 \\
 0 & 1 & 0 & 0 \\
0 & 0 & 1 & 0 \\
0 & 0 & 0 & 1
\end{array}
\right).
\end{eqnarray*}
Applying \thmref{thm main3} gives $\det(T_{n}) = c^{n} \cdot (u_{n+3} v_{n+2}-u_{n+2} v_{n+3})$.

As done in the previous section, one can use the formulas for the homogeneous linear recurrence relation $u_n$ and $v_n$ to derive a formula of $\det(T_{n})$ in terms of the roots of the characteristic polynomial of $u_n$ and $v_n$, which is nothing but $ch_C(x)$. Instead, we used the Jordan canonical forms in our computations below:

Note that $C$ is an irreducible Hessenberg matrix, and that both the minimal and the characteristic polynomials of such matrices are the same (see \cite[page 325]{SB}).
Alternatively, the reader may use a computer algebra system such as Mathematica \cite{Matem}) to check that the minimal polynomial of $C$ can not be of degree $1$, $2$ or $3$.
This makes some restrictions on the types of Jordan canonical forms, since the multiplicity of an eigenvalue $\lambda$ in the minimal polynomial is the size of the largest Jordan block corresponding to $\lambda$.
We have the following five cases.
We used Mathematica \cite{Matem} for the details of these computations.

\textbf{Case I:} $ch_C(x)=(x-\lambda_1)(x-\lambda_2)(x-\lambda_3)(x-\lambda_4)$, where the eigenvalues $\lambda_i$'s for $i=1, \cdots 4$ are distinct. Then, $J=\text{diag}(J_{(1)}(\lambda_1), J_{(1)}(\lambda_2), J_{(1)}(\lambda_3), J_{(1)}(\lambda_4))$, and we have
$\displaystyle{\det(T_{n})=c^n\frac{E}{D}}$, where
\begin{equation}\label{eqn det case I}
\begin{split}
E&=\left(\lambda_2-\lambda_3\right) \left(\lambda_1-\lambda_4\right) \left(\left(\lambda_2 \lambda_3\right)^{n+2}+\left(\lambda_1 \lambda_4\right)^{n+2}\right)-\left(\lambda_1-\lambda_3\right) \left(\lambda_2-\lambda_4\right) \cdot \\
& \quad \big(\left(\lambda_1 \lambda_3\right)^{n+2} +\left(\lambda_2 \lambda_4\right)^{n+2}\big)+\left(\lambda_1-\lambda_2\right) \left(\lambda_3-\lambda_4\right) \left(\left(\lambda_1 \lambda_2\right)^{n+2}+\left(\lambda_3 \lambda_4\right)^{n+2}\right),\\
D&=\left(\lambda_1-\lambda_2\right) \left(\lambda_1-\lambda_3\right) \left(\lambda_1-\lambda_4\right) \left(\lambda_2-\lambda_3\right) \left(\lambda_2-\lambda_4\right) \left(\lambda_3-\lambda_4\right).
\end{split}
\end{equation}

\textbf{Case II:} $ch_C(x)=(x-\lambda_1)^2(x-\lambda_2)(x-\lambda_3)$, where the eigenvalues $\lambda_1$, $\lambda_2$ and $\lambda_3$ are distinct. Then, $J=\text{diag}(J_{(2)}(\lambda_1), J_{(1)}(\lambda_2), J_{(1)}(\lambda_3))$, and we have
$\displaystyle{\det(T_{n})=c^n\frac{E}{D}}$, where
\begin{equation}\label{eqn det case II}
\begin{split}
E&=\lambda_1^{1 + n} (\lambda_1^{3 + n} (\lambda_2 - \lambda_3) +
    \lambda_2^{2 + n} (\lambda_1 (-(2 + n) \lambda_1 + \lambda_2 +
          n \lambda_2) + ((3 + n)\lambda_1 \\
          & \quad- (2 + n) \lambda_2) \lambda_3))+
 \lambda_3^{2 + n} (\lambda_2^{2 + n} (\lambda_2 - \lambda_3) +
    \lambda_1^{1 + n} ((2 + n) \lambda_1^2
    \\ & \quad + (2 + n) \lambda_2 \lambda_3
     -\lambda_1 ((3 + n) \lambda_2 + \lambda_3 + n \lambda_3))),\\
D&=\left(\lambda_1-\lambda_2\right)^2 \left(\lambda_1-\lambda_3\right)^2 \left(\lambda_2-\lambda_3\right).
\end{split}
\end{equation}
Note that formula of $\frac{E}{D}$ using \eqnref{eqn det case II} is the limiting value of $\frac{E}{D}$ using \eqnref{eqn det case I} as $\lambda_4 \longrightarrow \lambda_1$.

\textbf{Case III:} $ch_C(x)=(x-\lambda_1)^2(x-\lambda_2)^2$, where the eigenvalues $\lambda_1$ and $\lambda_2$  are distinct.
Then, $J=\text{diag}(J_{(2)}(\lambda_1), J_{(2)}(\lambda_2))$, and we have
$\displaystyle{\det(T_{n})=c^n\frac{E}{D}}$, where
%
\begin{equation}\label{eqn det case III}
\begin{split}
E&=\lambda_1^{4 + 2 n} - (2 + n)^2 (\lambda_1 \lambda_2)^{1 + n}(\lambda_1^2+ \lambda_2^2) +
 2 (3 + 4 n + n^2) (\lambda_1 \lambda_2)^{2 + n} + \lambda_2^{4 + 2 n},\\
D&=\left(\lambda_1-\lambda_2\right)^4.
\end{split}
\end{equation}
In this case, the formula of $\frac{E}{D}$ using \eqnref{eqn det case III} is the limiting value of $\frac{E}{D}$ using \eqnref{eqn det case II} as $\lambda_3 \longrightarrow \lambda_2$.

%

\textbf{Case IV:} $ch_C(x)=(x-\lambda_1)^3(x-\lambda_2)$, where the eigenvalues $\lambda_1$ and $\lambda_2$  are distinct.
Then, $J=\text{diag}(J_{(3)}(\lambda_1), J_{(1)}(\lambda_2))$, and we have
$\displaystyle{\det(T_{n})=c^n\frac{E}{D}}$, where
%
\begin{equation}\label{eqn det case IV}
\begin{split}
E&=(2 + n) \lambda_1^n \big((1 + n) (\lambda_1^{3 + n}-\lambda_2^{3 + n}) - (3 + n) \lambda_1^{2 + n}
     \lambda_2 + (3 + n) \lambda_1 \lambda_2^{2 + n}\big),\\
D&=2 \left(\lambda_1-\lambda_2\right)^3.
\end{split}
\end{equation}
Note that the formula of $\frac{E}{D}$ depending on \eqnref{eqn det case IV} is the limiting value of $\frac{E}{D}$ relying on \eqnref{eqn det case II} as $\lambda_3 \longrightarrow \lambda_1$.

\textbf{Case V:} $ch_C(x)=(x-\lambda_1)^4$. Then, $J=J_{(4)}(\lambda_1)$, and we have
$$\det(T_{n})=\frac{c^n}{12} (n+3) (n+2)^2 (n+1)\lambda_1^{2n}.$$
Note that this formula of $\det(T_{n})$ is the limiting value of $\det(T_{n})$ given in case IV as $\lambda_2 \longrightarrow \lambda_1$.

Next, we provide examples for each of these five cases.

\textbf{Example I:}
If $a=101$, $b=-17$, $c=1$, $d=-247$, $e=210$, then
we have
$ch_C(x)=(x-2)(x-3)(x-5)(x-7)$,
and
for any integer $n \geq 4$,
$$\det(T_{n})=\frac{1}{120} (-6 \cdot 10^{n+2}+5 \cdot 15^{n+2}+6^{n+2}-6 \cdot 21^{n+2}+5 \cdot 14^{n+2}+35^{n+2}).$$

\textbf{Example II:}
If $a=17$, $b=-7$, $c=1$, $d=-17$, $e=6$, then
we have
$ch_C(x)=(x-1)^2(x-2)(x-3)$,
and
for any integer $n \geq 4$,
$$\det(T_{n})=\frac{1}{4} ( 2^{n+2} (2n+3)-3^{n+2}(2n+5)+6^{n+2}+1 ).$$

\textbf{Example III:}
If $a=37$, $b=-10$, $c=1$, $d=-60$, $e=36$, then
we have
$ch_C(x)=(x-2)^2(x-3)^2$,
and
for any integer $n \geq 4$,
$$\det(T_{n})=4^{n+2} + 9^{n+2} - (n(n+4)+16) 6^{n+1}.$$

\textbf{Example IV:}
If $a=30$, $b=-9$, $c=1$, $d=-44$, $e=24$, then
we have
$ch_C(x)=(x-2)^3(x-3)$,
and
for any integer $n \geq 4$,
$$\det(T_{n})=2^{n-1} (n+2) (3^{n+2} (n-3) + 2^{n+2} (n+7)).$$

\textbf{Example V:}
If $a=24$, $b=-8$, $c=1$, $d=-32$, $e=16$, then we have $ch_C(x)=(x-2)^4$,
and
for any integer $n \geq 4$,
$$\det(T_{n})=\frac{4^{n-1}}{3} (n+3) (n+2)^2 (n+1).$$


It is straightforward to implement the algorithm we outlined to obtain similar closed form formulas for other small values of $r$, $s$ (and so $k$) by using Mathematica \cite{Matem} or some other computer programs having symbolic capabilities. However, the formulas for $k \geq 6$ are not short enough to include here.


\end{document}